\documentclass[12pt,a4paper]{article}
\usepackage[matrix,arrow,ps,color,line,curve,frame,all]{xy} 
\usepackage[left=3cm,top=3cm,right=3cm,bottom=3cm]{geometry}
\usepackage[english]{babel}
\usepackage{hyperref}
\usepackage{amsfonts}
\usepackage{amssymb}
\usepackage{amsmath}
\usepackage{amsthm}
\usepackage{parskip}
\usepackage{dsfont}

\newtheoremstyle{standard}{10pt}{3pt}{\itshape}{}{\bfseries}{.}{.5em}{}
\theoremstyle{standard}
\newtheorem{lemma}{Lemma}[subsection]
\newtheorem{prop}[lemma]{Proposition}
\newtheorem{thm}[lemma]{Theorem}
\newtheorem{cor}[lemma]{Corollary}
\newtheorem*{hal}{Theorem}

\newtheoremstyle{definition}{10pt}{3pt}{}{}{\bfseries}{.}{.5em}{}
\theoremstyle{definition}
\newtheorem{defi}[lemma]{Definition}  
  
\newtheorem{ex}[lemma]{Example}
\newtheorem{rem}[lemma]{Remark}

\DeclareMathOperator{\Q}{\mathbf{Qcoh}}
\DeclareMathOperator{\Cat}{\mathbf{Cat}}
\DeclareMathOperator{\M}{\mathbf{Mod}}

\DeclareMathOperator{\Alg}{\mathbf{Alg}}
\DeclareMathOperator{\Sch}{\mathbf{Sch}}
\DeclareMathOperator{\Stack}{\mathbf{Stack}}
\DeclareMathOperator{\zRing}{\mathbf{2-Ring}}

\DeclareMathOperator{\Spec}{Spec}
\DeclareMathOperator{\SPEC}{\mathbf{Spec}}
\DeclareMathOperator{\supp}{supp}
\DeclareMathOperator{\End}{End}
\DeclareMathOperator{\END}{\underline{\End}}
\DeclareMathOperator{\id}{id}
\DeclareMathOperator*{\colim}{colim}
\DeclareMathOperator{\op}{op}
\DeclareMathOperator{\Hom}{Hom}
\DeclareMathOperator{\lax}{lax}

\renewcommand{\O}{\mathcal{O}}
\DeclareMathAlphabet\mathbfcal{OMS}{cmsy}{b}{n}
\newcommand{\OO}{\mathbfcal{O}}

\begin{document}

\title{Tensor functors between categories of quasi-coherent sheaves}
\author{Martin Brandenburg\footnote{WWU M\"unster, \texttt{brandenburg@uni-muenster.de}} ~and Alexandru Chirvasitu\footnote{University of California, Berkeley, \texttt{chirvasitua@math.berkeley.edu}}}
\date{{\small \today}}
\maketitle

\begin{abstract}
\noindent For a quasi-compact quasi-separated scheme $X$ and an arbitrary scheme $Y$ we show that the pullback construction $f \mapsto f^*$ implements an equivalence between the discrete category of morphisms $Y \to X$ and the category of cocontinuous tensor functors $\Q(X) \to \Q(Y)$. This is an improvement of a result by Lurie and may be interpreted as the statement that algebraic geometry is $2$-affine. Moreover, we prove the analogous version of this result for Durov's notion of generalized schemes over $\mathds{F}_1$.
\end{abstract}

\tableofcontents


\section{Introduction}

For every morphism of schemes $f : Y \to X$ we have an associated pullback functor $f^* : \Q(X) \to \Q(Y)$ between categories of quasi-coherent sheaves. It preserves direct sums, cokernels and tensor products, i.e. it is a \emph{cocontinuous symmetric monoidal functor}. In this paper, we are concerned with the converse: is every cocontinuous symmetric monoidal functor $F : \Q(X) \to \Q(Y)$ induced by a morphism $Y \to X$? Or more precisely, does $f\mapsto f^*$ implement an equivalence between $\Hom_{\Sch}(Y,X)$ (regarded as a discrete category) and the category $\Hom_{c\otimes}\bigl(\Q(X),\Q(Y)\bigr)$ of all cocontinuous symmetric monoidal functors and monoidal natural transformations? 

Apart from being interesting in its own right, this question arises naturally as part of the discussion on ``2-algebraic geometry" in \cite{ChJo}. In that paper, a notion of commutative 2-ring is introduced. These are symmetric monoidal categories satisfying some extra technical conditions; the important thing for us is that categories of the form $\Q(X)$ are examples of commutative 2-rings, and one would like to conclude that $X\mapsto \Q(X)$ is fully faithful, and hence one can recover 1-algebraic geometry as affine 2-algebraic geometry. 

The problem also arises in the course of Jacob Lurie's work on Tannaka reconstruction for stacks. \cite[Theorem 5.11]{Lur} solves a modified version of this problem (for so-called geometric stacks, rather than just schemes): instead of the full $\Hom_{c\otimes}\bigl(\Q(X),\Q(Y)\bigr)$, Lurie's result involves the subcategory whose objects are functors satisfying a technical condition referred to as \emph{tameness} (cf. \cite[Definition 5.9]{Lur}), and whose morphisms are monoidal {\it iso}morphisms. Essentially, tameness means preservation of faithful flatness. This is a rather awkward global condition which seems to be very hard to check, so it is tempting to try to avoid it. 

Our aim, then, is to answer the question affirmatively for schemes satisfying reasonable conditions, but with no additional assumptions on our symmetric monoidal functors. The main result of the paper (see \ref{main}) is:

\begin{hal}
Let $X$ be a quasi-compact quasi-separated scheme, and $Y$ an arbitrary scheme. Then, the functor $f\mapsto f^*$ implements an equivalence
\[\Hom(Y,X)\simeq \Hom_{c\otimes}\bigl(\Q(X),\Q(Y)\bigr).\]
\end{hal}
 
This improves on the corresponding result on projective schemes, obtained in \cite{Br} by rather different methods. In the language of that paper, our result is that every quasi-compact quasi-separated scheme is tensorial. It is unclear whether or not the same is true for quasi-compact quasi-separated (algebraic) stacks, but certain partial results can be obtained (although we will not discuss stacks much in this paper); we mention, for example, that the result is not very difficult to prove for classifying stacks of finite groups. Recently, Daniel Schäppi (\cite{Sch1,Sch2}) has proven the corresponding equivalence if $X$ is an Adams stack (i.e. a geometric stack with the strong resolution property) and even classified those monoidal categories which are of the form $\Q(X)$ for some Adams stack $X$ using a suitable generalization of Tannakian categories.
 
The paper is organized as follows:

Section 2 reviews basics on monoidal categories and universal cocompletions. This will in particular settle Theorem \ref{main} in the special case when $X$ is affine. We also make some preparations concerning immersions. 

Section 3 contains the proof of our Theorem \ref{main}. It is broken down into two steps: First, $Y$ is (the spectrum of) a local ring, and finally, an arbitrary scheme. We remark that the proof is entirely elementary and self-contained, in the sense that it only uses algebraic geometry as developed in EGA I \cite{EGAI}. We also include a subsection outlining a general picture of what ``categorified algebraic geometry'' should look like, and how the results obtained in this paper fit into that framework. 

Finally, in Section 4 we argue that the main results and their proofs carry over to algebraic geometry in the context of generalized schemes as developed in \cite{Du}, with only minor modifications.

\subsection*{Acknowledgements:} We would like to thank Jacob Lurie and Daniel Sch\"appi for fruitful discussions on the contents of \cite{Lur} and \cite{Sch1,Sch2} respectively, as well as the anonymous referee for suggestions that have improved the substance and presentation of the paper.


\section{Preliminaries on tensor categories}


\subsection{Definition of tensor categories}

All rings and algebras under consideration are commutative and unital. A \emph{tensor category} is a category together with a tensor product which is unital, associative and symmetric up to compatible isomorphisms. These are called ACU $\otimes$-categories in (\cite[2.4]{SaR}), and are commonly known as symmetric monoidal categories (\cite[XI.1]{Mac}). In addition, tensor categories are assumed to be $R$-linear for some fixed ring $R$: this means that the underlying category is $R$-linear and the tensor product is $R$-linear in both variables (\cite[0.1.2]{SaR}). Tensor functors are understood to be strong, i.e. they respect the tensor structure up to a canonical \emph{iso}morphism (as in \cite[4.1.1, 4.2.4]{SaR}). For $R$-linear tensor categories $C,D$ we denote by $\Hom_{\otimes/R}(C,D)$ (or by $\Hom_{\otimes}(C,D)$ if $R$ is clear from the context) the category of all $R$-linear tensor functors $C \to D$. Morphisms in this category are tensor natural transformations, that is natural transformations which are compatible with the tensor structure (\cite[4.4.1]{SaR}). Thus we obtain the 2-category of $R$-linear tensor categories. The unit of a tensor category $C$ is usually denoted by $1_C$. The a priori noncommutative $R$-algebra $\End(1_C)$ turns out to be commutative by a variation of the Eckmann-Hilton argument (\cite[1.3.3.1]{SaR}). $\Alg(R)$ denotes the category of $R$-algebras, and $\M(R)$ that of \emph{right} $R$-modules.
  
By a \emph{cocomplete} tensor category we mean one whose underlying category is cocomplete (i.e. has all small colimits), and such that the tensor product is cocontinuous in each variable. This means that for all objects $X$ and all small diagrams $\{Y_i\}$, the canonical morphism
\[\colim\limits_i (X \otimes Y_i) \to X \otimes \colim\limits_i Y_i\]
is an isomorphism, and similarly for the other variable (this also follows by symmetry). For discrete diagrams this is just the categorified distributive law
\[\bigoplus_i (X \otimes Y_i) = X \otimes \bigoplus_i Y_i.\]
Therefore, we can think of $R$-linear cocomplete tensor categories as categorified $R$-algebras and might call them $R$-2-algebras. In fact, Chirvasitu and Johnson-Freyd call them $2$-rings (\cite[2.3.1]{ChJo}), dropping the $\M(R)$-enrichment and assuming presentability of the underlying category.
  
If $S$ is an $R$-algebra, then $\M(S)$ is an $R$-linear cocomplete tensor category. The tensor product is the usual tensor product of $S$-modules and the unit is $S$. More generally, if $X$ is an $R$-scheme, then its category of quasi-coherent modules $\Q(X)$ is an $R$-linear cocomplete tensor category with tensor product $\otimes_X$ and unit $\O_X$. This is our main example. Various reconstruction theorems such as the classical one by Gabriel for noetherian schemes (\cite{Gab}), by Rosenberg for quasi-separated schemes (\cite{Ros},\cite{Br2}) and recently by Lurie for geometric stacks (\cite{Lur}) suggest that all the information of (a nice) $X$ is already encoded in this $2$-algebra $\Q(X)$ and therefore we can think of usual algebraic geometry as $2$-affine (see also \cite[1.2]{ChJo}).
 
If $C,D$ are $R$-linear cocomplete tensor categories, we denote by $\Hom_{c\otimes/R}(C,D)$ the category of all cocontinuous $R$-linear tensor functors $C \to D$; if $C,D$ are just cocomplete categories, we denote by $\Hom_{c}(C,D)$ the category of cocontinuous functors. For example, every morphism $f : Y \to X$ of $R$-schemes induces a cocontinuous $R$-linear tensor functor $f^* : \Q(X) \to \Q(Y)$. As we have already mentioned in the introduction, the main purpose of this paper is to show that every cocontinuous tensor functor arises in this way.


\subsection{Universal cocompletion}

Fix a small $R$-linear category $C$. The category of presheaves
\[\widehat{C} := \Hom_{/R}\bigl(C^{\op},\M(R)\bigr)\]
is cocomplete and we have the Yoneda embedding $C \to \widehat{C}$. In fact, it is the universal cocompletion of $C$:

\begin{prop}
For a cocomplete $R$-linear category $D$, the Yoneda embedding induces an equivalence of categories $\Hom_{c/R}\bigl(\widehat{C},D\bigr) \simeq \Hom_{/R}(C,D)$.
\end{prop}

\begin{proof}
See \cite[4.4]{Kel1} for a proof in the context of general enriched categories. The crux is that every presheaf is a canonical colimit of representable functors (\cite[III.7]{Mac}).
\end{proof}

If $C$ is a small $R$-linear tensor category, then we may extend the tensor product $C \times C \to C$ to a bicocontinuous functor $\widehat{C} \times \widehat{C} \to \widehat{C}$ (called \emph{Day convolution}); just use the universal property of $\widehat{C}$ in both factors. Then, $\widehat{C}$ becomes a cocomplete $R$-linear category and in fact, has the following universal property:

\begin{prop}
For a cocomplete $R$-linear tensor category $D$, the Yoneda embedding induces an equivalence of categories
\[\Hom_{c\otimes/R}\bigl(\widehat{C},D\bigr) \simeq \Hom_{\otimes/R}(C,D).\]
\end{prop}

\begin{proof}
As sketched above, this is a direct consequence of the previous proposition. See \cite{ImKe} for a proof for general enriched tensor categories.
\end{proof}

\begin{prop} \label{modtens}
Let $A$ be an $R$-algebra. Then, $\M(A)$ is an $R$-linear cocomplete tensor category satisfying the following universal property: for every $R$-linear cocomplete tensor category $C$, we have an equivalence of categories
\[\Hom_{c\otimes/R}\bigl(\M(A),C\bigr) \simeq \Hom_{\Alg(R)}\bigl(A,\End(1_C)\bigr)\]
\end{prop}

Here the right hand side is a set considered as a discrete category.

\begin{proof}
We may consider $A$ as an $R$-linear tensor category with just one object and endomorphism algebra $A$. The tensor product of endomorphisms is the multiplication of $A$. Then, $R$-linear tensor functors $A \to C$ correspond to $R$-algebra homomorphisms $A \to \End(1_C)$. Unpacking the definition of $\widehat{A}$, it is clear that it is exactly $\M(A)$, with its usual tensor structure. Now apply the previous proposition.
\end{proof}

\begin{cor} \label{afftens}
Let $X,Y$ be $R$-schemes, where $X$ is is affine. Then, $f \mapsto f^*$ is an equivalence of categories $\Hom_R(Y,X) \simeq \Hom_{c\otimes/R}\bigl(\Q(X),\Q(Y)\bigr)$.
\end{cor}

\begin{proof}
Let $X = \Spec(A)$ for some $R$-algebra $A$. Then we have an equivalence
\begin{multline*}
\Hom_R(Y,X)  \simeq  \Hom_{\Alg(R)}\bigl(A,\Gamma(\O_Y)\bigr) = \Hom_{\Alg(R)}\bigl(A,\End(\O_Y)\bigr)\\
\simeq \Hom_{c\otimes/R}\bigl(\M(A),\Q(Y)\bigr) \simeq \Hom_{c\otimes/R}\bigl(\Q(X),\Q(Y)\bigr).
\end{multline*}
It follows from the constructions that this composition is exactly $f \mapsto f^*$.
\end{proof}


\subsection{Immersions}

\begin{defi}
Recall that a \textit{lax tensor functor} $F : D \to E$ is a functor equipped with morphisms $u_F : 1_E \to F(1_D)$ and $(c_F)_{X,Y} : F(X) \otimes F(Y) \to F(X \otimes Y)$ for $X,Y \in D$, which are compatible in a suitable sense (\cite[XI.2]{Mac}). Thus, in contrast to a (strong) tensor functor, we don't require these morphisms to be invertible. An \textit{oplax tensor functor} $F : D \to E$ is just a lax tensor functor $F : D^{\op} \to E^{\op}$, which means that $F$ comes equipped with natural morphisms $F(1_D) \to 1_E$ and $F(X \otimes Y) \to F(X) \otimes F(Y)$. Together with tensor natural transformations the lax tensor functors constitute a category $\Hom_{\lax}(D,E)$.
\end{defi}

\begin{rem}
A typical example arises from adjunctions: Assume we are given a tensor functor $G : E \to D$, whose underlying functor has a right adjoint $F : D \to E$, with unit $\eta : \id_E \to FG$ and counit $\varepsilon : GF \to \id_D$. Then $F$ becomes a lax tensor functor as follows: We define $u_F$ as the composition
\[\xymatrix{1_E \ar[r]^-{\eta} & F(G(1_E)) \ar[r]^-{u_G^{-1}} & F(1_D)}\]
and define $c_F$ as the composition
\[\xymatrix@C=14pt{F(X) \otimes F(Y) \ar[r]^-{\eta} & F(G(F(X) \otimes F(Y))) \ar[r]^-{c_G^{-1}} & F(G(F(X)) \otimes G(F(Y))) \ar[r]^-{\varepsilon \otimes \varepsilon} & F(X \otimes Y).}\]
\end{rem}

\begin{lemma} \label{laxadjoint}
In the above situation, $F$ is indeed a lax tensor functor. Besides, the unit $\eta : \id_E \to FG$ as well as the counit $\varepsilon : GF \to \id_D$ are tensor natural transformations.
\end{lemma}

\begin{proof}
This is an exercise in diagram chasing. On the other hand, it is a special case of a Theorem by Kelly about doctrinal adjunctions (\cite[1.4]{Kel2}).
\end{proof}

\begin{ex} \label{directimage}
If $f : Y \to X$ is a quasi-compact, quasi-separated morphism, then the direct image functor $f_*$ preserves quasi-coherence (\cite[6.7.1]{EGAI}), thus restricts to a functor $f_* : \Q(Y) \to \Q(X)$ which is right adjoint to the tensor functor $f^* : \Q(X) \to \Q(Y)$. By the discussion above, $f_*$ becomes a lax tensor functor.

In the special case that $i : Y \hookrightarrow X$ is a quasi-compact immersion of schemes (for example a closed immersion or a quasi-compact open immersion), observe that $i_*$ is fully faithful, which implies that the counit $i^* i_* \to \id$ is an isomorphism. The following proposition deals with this kind of situation and will become important later: 
\end{ex}

\begin{defi} \label{localdef}
Let $C,D,E$ be tensor categories and $i = (i^*,i_*,\eta,\varepsilon)$ an adjunction, where $i^* : E \to D$ is a tensor functor and, as above, $i_* : D \to E$ becomes a lax tensor functor. Assume that the counit $\varepsilon : i^* i_* \to \id_D$ is an isomorphism. A functor $F : E \to C$ is called \emph{$i$-local} if $F \eta : F \to F i_* i^*$ is an isomorphism. If the adjunction is induced by a quasi-compact immersion $i : Y \hookrightarrow X$ as above, we also say \emph{$Y$-local} instead of $i$-local.
\end{defi}
 
\begin{prop}
In the situation of Definition \ref{localdef} we have: 
\begin{enumerate}
 \item $F$ is $i$-local if and only if $F$ maps every morphism $\phi$, such that $i^* \phi$ is an isomorphism, to an isomorphism.  
 \item For every lax tensor functor $G : D \to C$ the lax tensor functor $G i^* : E \to C$ is $i$-local. Conversely, to every $i$-local lax tensor functor $F : E \to C$ we may associate the lax tensor functor $F i_* : D \to C$. This establishes an equivalence of categories
\[\Hom_{\lax}(D,C) \cong \{F \in \Hom_{\lax}(E,C)  \text{ $i$-local}\}.\]
\item This restricts to an equivalence of categories
\[\Hom_{\otimes}(D,C) \cong \{F \in \Hom_{\otimes}(E,C)  \text{ $i$-local}\}.\]
\item If $C,D,E$ are cocomplete tensor categories, then it even restricts to an equivalence of categories
\[\Hom_{c\otimes}(D,C) \cong \{F \in \Hom_{c\otimes}(E,C)  \text{ $i$-local}\}.\]
\end{enumerate}
\end{prop}

\begin{proof}
$1.$ First note that $i^* \eta : i^* \to i^* i_* i^*$ is an isomorphism, since it is right inverse to the isomorphism $\varepsilon i^*$. This already shows one direction. Now assume that $F$ is $i$-local and consider a morphism $\phi : M \to N$ in $E$ such that $i^* \phi$ is an isomorphism. The naturality of $\eta$ with respect to $\phi$ yields the commutative diagram
\[\xymatrix@C=35pt{F(M) \ar[r]^{F(\phi)} \ar[d]_{F(\eta_M)} & F(N) \ar[d]^{F(\eta_N)} \\ F(i_* i^* M) \ar[r]_{F (i_* i^* \phi)} & F(i_* i^* N).}\]
Since $F$ is $i$-local, the vertical arrows are isomorphisms. The bottom arrow is an isomorphism since $i^* \phi$ is one. Thus also the top arrow $F(\phi)$ is an isomorphism.

$2.$ Let $G : D \to C$ be a lax tensor functor. Since $i^* \eta$ is an isomorphism, we see that $G i^* : E \to C$ is $i$-local. Conversely, if $F : E \to C$ is an $i$-local lax tensor functor, then $F i_* : D \to C$ is a lax tensor functor. Of course, the same works for tensor natural transformations. We obtain functors
\[\xymatrix@C=40pt{\Hom_{\lax}(D,C) \ar@<1.2ex>[r]^-{- \circ i^*}  &  \ar@<1.2ex>[l]^-{- \circ i_*} \{F \in \Hom_{\lax}(E,C)  \text{ $i$-local}\}.}\]
Let us show that they are pseudo-inverse to each other. Given a lax tensor functor $G : D \to C$ we have a natural isomorphism of functors $G \varepsilon : G i^* i_* \to G$, which is even a tensor natural transformation by Lemma \ref{laxadjoint}. Similarly, for a $i$-local lax tensor functor $F : E \to C$ the natural isomorphism $F \eta : F \to F i_* i^*$ is actually an isomorphism of lax tensor functors by the same lemma.
 
$3.$ Since $i^*$ is a tensor functor, we see that $G i^*$ is a tensor functor provided that $G$ is a tensor functor. Now assume that $F$ is an $i$-local tensor functor. We have to show that $F i_*$ is a tensor functor. The morphism $1_C \to (F i_*)(1_D)$ is defined as the composition $1_C \to F(1_E) \to F(i_* 1_D)$. The first morphism is an isomorphism since $F$ is a tensor functor. The second one is $F$ applied to $\eta_{1_E} : 1_E \to i_* 1_D$, hence also an isomorphism. Now let $M,N  \in D$. The morphism $(F i_*)(M) \otimes (F i_*)(N) \to (F i_*)(M \otimes N)$ is defined as the composition $F(i_*(M)) \otimes F(i_*(N)) \to F(i_* M \otimes i_* N) \to F(i_* (M \otimes N))$, where the first morphism is an isomorphism since $F$ is a tensor functor and the second one is $c : i_* M \otimes i_* N \to i_* (M \otimes N)$ mapped by $F$. But $i^* c$ is an isomorphism. By $1.$ above, $F$ maps $c$ to an isomorphism.

$4.$ Since $i^*$ has a right adjoint, it is cocontinuous. We see that $G i^*$ is a cocontinuous tensor functor provided that $G$ is a cocontinuous tensor functor. Now assume that $F$ is an $i$-local cocontinuous tensor functor. We have to show that $F i_*$ is cocontinuous. This works as before: For a diagram $\{M_j\}$ in $D$, the canonical morphism $\colim_j (F i_*)(M_j) \to (F i_*)(\colim_j M_j)$ factors as the isomorphism $\colim_j F(i_* M_j) \cong F(\colim_j i_* M_j)$ followed by $F$ applied to the canonical morphism $\colim_i i_* M_j \to i_* (\colim_i M_j)$, which is clearly an isomorphism after applying $i^*$. Thus $F$ maps it to an isomorphism, and we're done.
\end{proof}

\begin{cor} \label{immersion}
Let $i : Y \to X$ be a quasi-compact immersion of schemes. Then, for every cocomplete tensor category $C$, the functors $i_*$ and $i^*$ induce an equivalence of categories
\[\Hom_{c\otimes}\bigl(\Q(Y),C\bigr) \cong \{F \in \Hom_{c\otimes}\bigl(\Q(X),C\bigr) \text{ $Y$-local}\}\]
\end{cor}

\begin{cor} \label{reduction} 
Let $F : \Q(X) \to \Q(Y)$ be a cocontinuous tensor functor. Assume that $i : U \to X$ is a quasi-compact immersion, where $U$ is affine. If $F$ is $i$-local, then $F$ is induced by a morphism.
\end{cor}

\begin{proof}
Combine corollaries \ref{immersion} and \ref{afftens} 
\end{proof}


\section{Proof of the Theorem}


\subsection{Full faithfulness}

We now embark on the proof of the theorem announced in the introduction. We have to prove that the functor $f\mapsto f^*$ is an equivalence, that is, it is fully faithful and essentially surjective. The following proposition takes care of the full faithfulness part. 
 
\begin{prop}\label{fullfaith}
Let $f,g : Y \to X$ be morphisms of schemes with $X$ quasi-separated, and let $\alpha:f^*\Rightarrow g^*$ be a tensor natural transformation. Then, $f=g$ and $\alpha=\id$. 
\end{prop}

\begin{proof}
Both $f=g$ and $\alpha=\id$ can be checked stalkwise on $Y$: They hold if and only if the corresponding statements hold for $f\circ i_y$ and $g\circ i_y$ for all points $y\in Y$, where $i_y:\Spec(\O_{Y,y})\to Y$ is the canonical map. As a consequence, we may assume that $Y$ is the spectrum of a local ring $S$. Its closed point will be denoted by $y$. 

Next we claim that $f(y)$ is a specialization of $g(y)$. Indeed, let $I \subseteq \O_X$ be the vanishing ideal sheaf of the closed subset $Z:=\overline{\{g(y)\}}$ of $X$. Since $\alpha$ induces a map of $S$-algebras $f^*(\O_X/I) \to g^*(\O_X/I)$, it follows that $\supp g^*(\O_X/I) \subseteq \supp f^*(\O_X/I)$. According to \cite[Chapitre 0, 5.2.4.1]{EGAI}, this means that $g^{-1}(Z) \subseteq f^{-1}(Z)$. Now $y \in g^{-1}(Z)$ implies $f(y) \in Z = \overline{\{g(y)\}}$, which proves our claim.

If $i : U \hookrightarrow X$ is an affine open neighborhood of $f(y)$, then it must contain $g(y)$, too. Since $Y$ is local, then both $f$ and $g$ factor through $U$, say $f=if'$ and $g=ig'$ for morphisms $f',g' : Y \to U$. Remark that $i$ is a quasi-compact immersion since $X$ is quasi-separated and that $f^* \cong f'^* i^*$ is $i$-local, similarly $g^*$. Now Corollary \ref{immersion} gives us a tensor natural transformation $\alpha' : f'^* \Rightarrow g'^*$ of tensor functors $\Q(U) \to \Q(Y)$ with $i^* \alpha'=\alpha$. But now are in the affine case and Corollary \ref{afftens} yields $f'=g'$ and $\alpha'=\id$. This implies $f=g$ and $\alpha=\id$, and we are done.
\end{proof}


\subsection{Good schemes}

The following notion will play a crucial role in the rest of the paper.

\begin{defi}
A scheme $Y$ is called \emph{good} if every cocontinuous tensor functor $\Q(X) \to \Q(Y)$, where $X$ is a quasi-compact quasi-separated scheme, is induced by some morphism $Y \to X$. A ring $A$ is called good if $\Spec(A)$ is good.
\end{defi}

\begin{rem}\label{remgood}
Since Proposition \ref{fullfaith} already provides the full faithfulness part of Theorem \ref{main}, $Y$ is good precisely when $f\mapsto f^*$ implements an equivalence \[\Hom(Y,X) \simeq \Hom_{c\otimes}\bigl(\Q(X),\Q(Y)\bigr)\] for every quasi-compact quasi-separated scheme $X$.   
\end{rem}

\begin{rem} \label{fpqc}
For an arbitrary scheme $X$, both the functor $\Hom(-,X)$ and the pseudofunctor $\Hom_{c\otimes}\bigl(\Q(X),\Q(-)\bigr)$ are stacks in the Zariski topology on the category of all schemes (even in the fpqc topology by descent theory \cite[2.55, 4.23]{Vis}, but we won't need that here). It is clear that the pullback construction gives a morphism of stacks
\[\Hom(-,X) \to \Hom_{c\otimes}\bigl(\Q(X),\Q(-)\bigr),\]
which we have already seen to be fully faithful if $X$ is quasi-separated. In particular, goodness is a Zariski local property. More explicitly, if $F$ is a cocontinuous tensor functor $\Q(X) \to \Q(Y)$  and $Y$ is covered by open subschemes $Y_i$ such that each restriction $F|_{Y_i} : \Q(X) \to \Q(Y_i)$ is induced by a unique morphism $Y_i \to X$, then full faithfulness implies that these morphisms glue to a morphism $Y \to X$ which induces $F$.
\end{rem}
 

\subsection{Local rings are good}

Let $(A,\mathfrak{m})$ be a local ring and $F : \Q(X) \to \M(A)$ be a cocontinuous tensor functor, where $X$ is quasi-compact and quasi-separated. 

\begin{defi}
Let $i : U \hookrightarrow X$ be an affine open subscheme. We call $F$ \textit{weakly $U$-local} if for every quasi-coherent ideal $I \subseteq \O_X$ with $I|_U = \O_U$, $F$ maps the inclusion $I \hookrightarrow \O_X$ to an isomorphism.
\end{defi}

\begin{lemma} \label{ideal}
If $I \subseteq \O_X$ is a quasi-coherent ideal such that $F$ maps the inclusion $I \hookrightarrow \O_X$ to an epimorphism, then $F$ even maps it to an isomorphism.
\end{lemma}

\begin{proof}
We may view $I$ as a non-unital commutative algebra object in the tensor category $\Q(X)$, such that the inclusion $I \hookrightarrow \O_X$ is an algebra homomorphism. Clearly, the diagram
\[\xymatrix@=14pt{I \otimes I \ar[rr] \ar[dr] && I \\ & \O_X \otimes I \ar[ur]_-{\cong} & }\]
commutes. Thus $B:=F(I)$ is a non-unital algebra object in $\M(A)$ together with a non-unital algebra homomorphism $f : B \to A$ such that $b c = f(b) c$ for all $b,c \in B$. By assumption, $f$ is surjective. Every element $e \in f^{-1}(1)$ will serve as a unit for $B$. Since $b = b e = f(b) e$, it follows that $f$ is also a monomorphism, and hence an isomorphism.
\end{proof}

\begin{lemma} \label{weak}
There is some affine open subscheme $i : U \hookrightarrow X$ such that $F$ is weakly $U$-local.
\end{lemma}

\begin{proof}
Choose some affine open covering $X = U_1 \cup \cdots \cup U_n$. If $F$ is not weakly $U_k$-local for every $1 \leq k \leq n$, then by Lemma \ref{ideal} there are quasi-coherent ideals $I_k \subseteq \O_X$ with $I_k |_{U_k} = \O_{U_k}$ such that $F(I_k \hookrightarrow \O_X)$ is not an epimorphism, thus factors through $\mathfrak{m} \subseteq A = F(\O_X)$. Then the same must be true for $F(\oplus_k I_k \to \O_X)$ and hence for $F(\sum_k I_k \hookrightarrow \O_X)$. This is a contradiction since $\sum_k I_k = \O_X$.
\end{proof}

\begin{lemma} \label{pairing}
Let $M,N,P \in \Q(X)$ and $\psi : M \otimes N \to P$ be a pairing. Assume we are given a quasi-coherent submodule $P' \subseteq P$. Define a submodule $M' \subseteq M$ by
\[\Gamma(U,M') = \bigl\{m \in \Gamma(U,M) : \forall V \subseteq U,\ \forall n \in \Gamma(V,N),\ \psi(m|_V \otimes n) \in \Gamma(V,P')\bigr\}.\]
If $N$ is of finite type, then $M'$ is quasi-coherent.
\end{lemma}

\begin{rem}
In other words, $M'$ is the largest submodule of $M$ whose pairing with $N$ by means of $\psi$ lands inside $P'\subseteq P$. We will denote it by $(P':N)_{\psi}$.
\end{rem}

\begin{proof}
Quasi-coherence is local on $X$, so we can assume that $X$ is affine, and hence there is an epimorphism $\O_X^n\to N$. By composing $\psi$ with the resulting epimorphism $M\otimes\O_X^n\to M\otimes N$, one gets a pairing $\psi':M\otimes\O_X^n\to P$. 

Now, $\psi':M\otimes\O_X^n\to P$ is essentially the same as a morphism $\psi':M\to P^n$, and $M'$ is nothing but the kernel of the morphism $M\stackrel{\psi'}{\rightarrow}P^n\to P^n/(P')^n$ of quasi-coherent sheaves. In conclusion, $M'$ is quasi-coherent. 
\end{proof}

\begin{lemma} \label{iso}
If $F$ is weakly $U$-local, then $F$ is $U$-local, that is: If $\phi : M \to N$ is a morphism in $\Q(X)$ which is an isomorphism on $U$, then $F$ maps $\phi$ to an isomorphism.
\end{lemma}

\begin{proof} Since $\phi$ factors as an epimorphism followed by a monomorphism, both of which are isomorphisms on $U$, it suffices to treat two cases:

\textit{1. $\phi$ is a monomorphism.} In this case $\phi$ may be treated as an inclusion $M \subseteq N$. Write $N$ as the filtered colimit of its quasi-coherent submodules $\{N_i\}$ of finite type (\cite[6.9.9]{EGAI}), and consider the pairings $\psi_i:\O_X\otimes N_i \to N$. 

Fixing an $i$, note that $I=(M:M+N_i)_{\psi_i}\subseteq \O_X$ equals $(M:N_i)_{\psi_i}$; by Lemma \ref{pairing}, $I$ is a quasi-coherent ideal. Thus we may replace $N_i$ by $N_i + M$ and assume $M \subseteq N_i$. By the very definition, $I$ is the largest ideal such that the multiplication $I \otimes N_i \to N$ factors through $M$. Observe that $I|_U = \O_U$, since $M|_U = N|_U$. The weak $U$-locality of $F$ implies that $F$ maps $I \hookrightarrow \O_X$ to an isomorphism. Consider the following commutative diagrams:
\[\xymatrix{I \otimes M \ar[r] \ar[d] & \O_X \otimes M \ar[d]^{\cong} \\ I \otimes N_i \ar[r] & M} ~ ~ ~ ~
\xymatrix{I \otimes N_i \ar[r] \ar[d] & \O_X \otimes N_i \ar[d]^{\cong} \\ M  \ar[r] & N_i}\]
Keeping in mind that $F(I\otimes M\to M)$ is an isomorphism (because $F(I\hookrightarrow \O_X)$ is one), the left hand diagram shows that $F$ maps $I \otimes M \to I \otimes N_i$ and hence also $M \hookrightarrow N_i$ to a split monomorphism. Similarly, the right hand diagram shows that $F(M\hookrightarrow N_i)$ is a split epimorphism. In conclusion, $F(M \hookrightarrow N_i)$ is an isomorphism for all $i$. Taking the colimit over $i$, it follows that $F(M \hookrightarrow N)$ is an isomorphism.

\textit{2. $\phi$ is an epimorphism.} Let $K = M \times_N M$ the difference kernel of $\phi$. It consists of pairs of sections $(m,m')$ such that $\phi(m)=\phi(m')$. Since $\phi$ is an epimorphism, it is the coequalizer of the two projections $p_1,p_2 : K \rightrightarrows M$. Then, $F(\phi)$ is the coequalizer of $F(p_1)$ and $F(p_2)$, so that it suffices to show $F(p_1)=F(p_2)$. If $i : M \to K$ is the diagonal homomorphism defined by $m \mapsto (m,m)$, we have $p_1 i = p_2 i = \id_M$. Since $\phi$ is an isomorphism on $U$, the same must be true for $i$. Now, $i$ is a split monomorphism, so the first case shows that $F(i)$ is an isomorphism. But then, $F(p_1)= F(i)^{-1} = F(p_2)$.
\end{proof}

\begin{thm} \label{local}
Every local ring is good.
\end{thm}

\begin{proof}
If $F : \Q(X) \to \M(A)$ is as above, we have seen in Lemma \ref{weak} that $F$ is weakly $U$-local for some $U$, hence $U$-local by Lemma \ref{iso}. Now use Corollary \ref{reduction}.
\end{proof}


\subsection{Every scheme is good}
  
\begin{prop} \label{map}
Let $X$ be a quasi-compact quasi-separated scheme and $Y$ an arbitrary scheme. Let $F : \Q(X) \to \Q(Y)$ be a cocontinuous tensor functor. Then, for every $y \in Y$, there is a local homomorphism $\O_{X,x} \to \O_{Y,y}$ for some $x\in X$, together with an isomorphism of tensor functors from $\Q(X)$ to $\M(\O_{Y,y})$:
\[F(M)_y \cong M_{x} \otimes_{\O_{X,x}} \O_{Y,y}.\]
Moreover, if we define $f(y) := x$, then the map $f : Y \to X$ is continuous.
\end{prop}

\begin{proof}
Let $y \in Y$. The pullback functor associated to the canonical morphism $\Spec(\O_{Y,y}) \to Y$ is given by $N \mapsto N_y$. Since the local ring $\O_{Y,y}$ is good (Theorem \ref{local}), the composition $F_y: \Q(X) \to \M(\O_{Y,y})$ is induced by a morphism $\Spec(\O_{Y,y}) \to X$. This factors as $\Spec(\O_{Y,y}) \to \Spec(\O_{X,x}) \to X$ for some $x \in X$ and some local homomorphism $\O_{X,x} \to \O_{Y,y}$ (\cite[2.5.3]{EGAI}), which proves the first part of the statement. 

For the second part, let $I \subseteq \O_X$ be a quasi-coherent ideal. Then $\O_X \to \O_X/I$ gets mapped by $F$ to an epimorphism, say $\O_Y \to \O_Y/J$ for some quasi-coherent ideal $J \subseteq \O_Y$. For $y \in Y$ and $x:=f(y)$, we have an isomorphism
\[(\O_Y/J)_y \cong (\O_X/I)_{x} \otimes_{\O_{X,x}} \O_{Y,y}.\]
Since $\O_{X,x} \to \O_{Y,y}$ is local, this shows that $\O_Y/J$ vanishes at $y$ if and only if $\O_X/I$ vanishes at $x$. We arrive at $f^{-1}(\supp(\O_X/I))=\supp(\O_Y/J)$, so $f$ is continuous.
\end{proof}

\begin{thm} \label{good}
Every scheme is good.
\end{thm}

\begin{proof}
Let $F : \Q(X) \to \Q(Y)$ be a cocontinuous tensor functor. By Proposition \ref{map} we may associate to it a continuous map $f : Y \to X$ (not a morphism yet). Cover $Y$ with open subsets which get mapped into affine open subsets of $X$. Since we may work locally on $Y$ (Remark \ref{fpqc}), we may assume that $f$ factors as $Y \to U \hookrightarrow X$ for some affine open subscheme $i : U \hookrightarrow X$. We claim that $F$ is $U$-local: Since isomorphisms can be checked stalkwise, it is enough to prove the claim for $F_y : \Q(X) \to \M(\O_{Y,y})$, where $y \in Y$. By Theorem \ref{local}, $F_y$ is induced by some morphism $g : \Spec(\O_{Y,y}) \to X$. As a continuous map $g$ is just a restriction of $f$ and therefore factors through $U$. Thus $F_y$ is $U$-local.
\end{proof}

In other words (cf. \ref{remgood}), we have proven:

\begin{thm}\label{main}
Let $X$ be a quasi-compact quasi-separated scheme, and $Y$ an arbitrary scheme. Then, the functor $f\mapsto f^*$ implements an equivalence
\[\Hom(Y,X)\simeq \Hom_{c\otimes}\bigl(\Q(X),\Q(Y)\bigr).\]
\end{thm}

In conclusion, the category of quasi-compact quasi-separated schemes (regarded as a 2-category with only identity 2-morphisms) embeds fully faithfully into the dual of the 2-category of cocomplete tensor categories. As explained in the introduction, this means that ordinary algebraic geometry is in a certain sense ``2-affine''; we will elaborate in the next section.

\begin{rem}
The same proof works when $Y$ is an arbitrary locally ringed space, in fact we have
\[\Hom(Y,X)\simeq \Hom_{c\otimes}\bigl(\Q(X),\M(Y)\bigr)\]
and every cocontinuous tensor functor $\Q(X) \to \M(Y)$ already maps into $\Q(Y)$ as defined in (\cite[Chapitre 0, 5.1.3]{EGAI}).
\end{rem}

\begin{cor}
Let $X$ be a quasi-compact quasi-separated scheme, and $Y$ an arbitrary locally ringed space. If $\Q_{\mathrm{fp}}(X)$ denotes the category of quasi-coherent modules of finite presentation on $X$, then $f \mapsto f^*$ implements an equivalence between the discrete category $\Hom(Y,X)$ and the category of right exact tensor functors $\Q_{\mathrm{fp}}(X) \to \Q_{\mathrm{fp}}(Y)$.
\end{cor}
 
\begin{proof}
This follows from Theorem \ref{main} and (\cite[3.11]{Br}).
\end{proof}

\begin{ex}
Here is a simple application of our main theorem. Let $X$, $Y$ be as above and let $g : Y \to X$ be a morphism \emph{of the underlying ringed spaces}, i.e. it is not required that the induced homomorphisms $g^\#_y : \O_{X,g(y)} \to \O_{Y,y}$ on local rings  are local. Nevertheless, we can still define a cocontinuous tensor functor $g^* : \Q(X) \to \Q(Y)$ (\cite[Chapitre 0, 5.1.4]{EGAI}). By Theorem \ref{main} there is a unique morphism $f : Y \to X$ \emph{of locally ringed spaces} such that $g^* \cong f^*$. Actually this also holds when $X$ is an arbitrary scheme. Let us outline three (equivalent) descriptions of $f$, the first one actually coming from the proof of the theorem:
\begin{enumerate}
\item By gluing it is enough to consider the case that $X=\Spec(A)$ is affine. Then $g$ yields on global sections a homomorphism $A \to \Gamma(Y,\O_Y)$, which corresponds to a morphism of locally ringed spaces $f : Y \to \Spec(A)$ (\cite[1.6.4]{EGAI}).
\item For $y \in Y$ the homomorphism $g^\#_x : \O_{X,g(y)} \to \O_{Y,y}$ pulls the maximal ideal of $\O_{Y,y}$ back to some prime ideal $\mathfrak{p}$ of $\O_{X,g(y)}$. The canonical morphism $\Spec(\O_{X,g(y)}) \to X$ (\cite[2.5.1]{EGAI}) maps it to a point $f(y)$ satisfying $g(y) \prec f(y)$. This defines a continuous map $f : Y \to X$. The maps on stalks are defined by $f^\#_y : \O_{X,f(y)} =  (\O_{X,g(y)})_{\mathfrak{p}} \to \O_{Y,y}$, which are local by construction. They extend to a sheaf homomorphism $\O_X \to f_* \O_Y$.
\item The forgetful functor $U$ from locally ringed spaces to ringed spaces has a right adjoint functor $R$ (\cite{Gil}). This adjunction induces a monad on the category of locally ringed spaces. Every scheme $X$ has a canonical module structure $R(U(X)) \to X$ with respect to this monad, which eventually comes down to the existence of canonical morphisms $\Spec(\O_{X,x}) \to X$. Now, if $Y$ is a locally ringed space, then a morphism of ringed spaces $g : U(Y) \to U(X)$ induces a morphism of locally ringed spaces
\[f : Y \to R(U(Y)) \to R(U(X)) \to X.\]
\end{enumerate}
\end{ex}


\subsection{Categorification}

In this section we outline the categorification process alluded to above. For basics on $2$-categories, we refer the reader to \cite[Section 9]{St}.

\begin{defi}\label{cat} \noindent \begin{enumerate}
\item A \textit{$2$-ring} is a cocomplete tensor category. The $2$-category of $2$-rings is denoted by $\zRing$.
\item The 1-opposite of $\zRing$ is called the $2$-category of \textit{$2$-affine schemes}.
\item A \textit{stack} is a pseudofunctor $\Sch^{\op} \to \Cat$ which satisfies effective descent with respect to the fpqc topology. Together with natural transformations and modifications, we obtain the $2$-category of stacks, denoted by $\Stack$.
\item Let $C$ be a $2$-ring. Its \textit{spectrum} $\SPEC(C)$ is defined by \[\SPEC(C)(X) = \Hom_{c\otimes}(C,\Q(X)),\ X \in \Sch.\] This is indeed a stack by descent theory for quasi-coherent modules (\cite[4.23]{Vis}), and this construction provides us with a 2-functor $\SPEC$ from $\zRing^{\op}$ to $\Stack$.
\item Let $F$ be a stack. The \textit{$2$-ring of regular $2$-functions} $\OO(F)$ is defined as $\Hom(F,\Q(-))$. This means that a regular $2$-function is given by functors $M_X : F(X) \to \Q(X)$ for every scheme $X$, which are compatible with base change, which means that there are coherent isomorphisms $f^* \circ M_X \cong M_Y \circ f^*$ for morphisms $f : Y \to X$. These isomorphisms also belong to the data. We obtain a 2-functor $\OO : \Stack \to \zRing^{\op}$.
\end{enumerate}
\end{defi}

\begin{rem}
If $F$ is an algebraic stack, then our definition of a $2$-regular function on $F$ coincides with the usual definition of a quasi-coherent module on $F$ (\cite[7.18]{V2}). Thus, we interpret quasi-coherent modules as categorified regular functions. As usual, $2$-regular functions are ``added'' and ``multiplied'' pointwise; similarly, colimits are computed pointwise. If $F$ is representable by a scheme $X$, then the $2$-Yoneda lemma implies $\OO(F) \simeq \Q(X)$. For a scheme $X$, we may therefore write $\OO(X)$ instead of $\Q(X)$.
\end{rem}

\begin{rem}
It will be obvious to the alert reader that in Definition \ref{cat} we have ignored some rather serious set-theoretic issues. For example, if $F$ is a stack as defined in part 3, i.e. a 2-functor defined on the category of \emph{all} schemes, then $\OO(F)$ as defined in part 5 might, in principle, be ``too large'' to be a category, in the sense that its set of objects is outside of some universe fixed beforehand (if one chooses Grothendieck universes as the set-theoretic framework for algebraic geometry). 

There are ways of getting around these difficulties, such as defining stacks only on small sites of schemes (rather than all schemes), and similarly bounding the ``size'' of our 2-rings in some sense. However, since the aim of this section is not to develop a theory in a rigorous and consistent manner but rather to outline a heuristic principle, we will not go into the details of how the issues could be resolved, and will continue to ignore them throughout the remainder of the section. 
\end{rem}

The following categorifies a well-known adjunction (\cite[1.6.3]{EGAI}):

\begin{prop}
If $F$ is a stack and $C$ is a $2$-ring, then there is a natural equivalence
\[\Hom(F,\SPEC(C)) \simeq \Hom(C,\OO(F)).\]
Thus, $\SPEC : \zRing^{\op} \to \Stack $ is right adjoint to $\OO : \Stack \to \zRing^{\op}$.
\end{prop}

\begin{proof}
This is entirely formal and an example of (higher) Isbell duality. For a stack $F$ and a scheme $X$, the component $F(X) \to \Hom(\OO(F),\OO(X))$ of the unit $\eta_F : F \to \SPEC(\O(F))$ is defined to be the obvious evaluation. Similarly, the counit $\varepsilon_C : C \to \OO(\SPEC(C))$ is given by evaluation.
\end{proof}

\begin{defi}
A stack $F$ is called \textit{$2$-affine} if the unit $\eta_F : F \to \SPEC(\OO(F))$ is an equivalence, i.e. for every scheme $X$ we have $F(X) \simeq \Hom(\OO(F),\OO(X))$. In \cite{Br} this notion was called \emph{tensorial}. A $2$-ring $C$ is called \textit{stacky} if the counit $C \to \OO(\SPEC(C))$ is an equivalence.
\end{defi}

Every adjunction restricts to an equivalence between its fixed points. Thus:

\begin{prop}
The functors $\SPEC$ and $\OO$ provide an anti-equivalence of $2$-categories between $2$-affine stacks and stacky $2$-rings. In particular, a $2$-affine scheme is completely determined by its $2$-ring of $2$-regular functions.
\end{prop}

Now our main Theorem \ref{main} reads as follows:

\begin{thm}
Every quasi-compact quasi-separated scheme is $2$-affine.
\end{thm}

In conclusion, although a general scheme $X$ cannot be recovered from its $1$-ring of regular functions $\Gamma(X,\O_X)$, it can nevertheless be recovered from its $2$-ring of $2$-regular functions $\OO(X)$, which is just its category of quasi-coherent modules. It would be interesting to know whether every Artin stack satisfying certain natural finiteness conditions (such as quasi-compactness and quasi-separatedness) is $2$-affine, or to have an intrinsic characterization for stacky $2$-rings; these and other similar problems arising naturally in this context will appear in the thesis of the first author.

We conclude this section with another example of the kind of notion that we believe should play an important role in categorified algebraic geometry. 

A morphism of quasi-compact quasi-separated schemes $f : X \to S$ is affine if and only if the canonical morphism $\SPEC(f_*(\O_X)) \to X$ over $S$ is an isomorphism. This motivates:

\begin{defi}
A morphism $f : X \to S$ of $2$-affine schemes is called $1$-affine if the corresponding cocontinuous tensor functor $F : C \to D$ satisfies the following condition: There is a right adjoint $F_*$ of $F$, and the canonical functor $\M(F_*(1_D)) \to D$, defined by $M \mapsto F(M) \otimes_{F(F_*(1_D))} 1_D$ (cf. \cite[Section 4]{Br}), is an equivalence of categories.
\end{defi}

Then it is easy to see that a morphism of quasi-compact quasi-separated schemes $f : X \to S$ is affine in the usual sense if and only if it is $1$-affine when considered as a morphism of $2$-affine schemes, i.e. upon identifying $f$ with its pullback functor $f^*$. As usual we also have an absolute notion:

\begin{defi}
Let $R$ denote our base ring. A $2$-affine scheme $X$ is called $1$-affine if the unique morphism from $X$ to the $2$-affine scheme corresponding to $\M(R)$ is $1$-affine. In other words, if $C$ is the $2$-ring corresponding to $X$, the canonical cocontinuous tensor functor $\M(\End_C(1_C))) \to C$ (cf. Proposition \ref{modtens}) is an equivalence of categories.
\end{defi}

Hence, a quasi-compact quasi-separated scheme $X$ is $1$-affine when considered as an $2$-affine scheme, i.e. upon identifying $X$ with $\OO(X) \cong \Q(X)$, if and only if $X$ is affine in the usual sense.


\section{Generalized schemes}


\subsection{Preliminaries and conventions}

Nikolai Durov (\cite{Du}) has developed a theory of \textit{generalized rings} and \textit{generalized schemes}, which contains, among others, the theory of classical schemes as developed in \cite{EGAI}, as well as the more recent theory of monoid schemes (\cite{CoCo}, \cite{Vez}, \cite{T}, \cite{PeLo}). Originally motivated by Arakelov geometry, Durov's theory provides an amazingly general framework for algebraic geometry. For a summary, the reader may consult \cite[Chapter 1]{Du} or \cite{Fre}.

In this section we would like to indicate how the results in the main portion of the paper (and especially Theorem \ref{main}) can be adapted to this general context. In fact, in most cases, the same proofs can be used. We will refrain from developing the algebraic geometry of generalized schemes in the style of \cite{EGAI}. Instead, we would like to argue briefly that essentially the same techniques that were applied for ordinary schemes will also work for generalized schemes, with only minor modifications and a minimal amount of effort. We do this in order to highlight the fact that not only ordinary algebraic geometry, but also ``absolute'' or $\mathds{F}_1$-algebraic geometry can be part of the categorification process glimpsed in the previous sections.  

In the following we will assume that the reader is familiar with Durov's theory, as in Chapters 5 and 6 of \cite{Du}. We will use the unary localization theory and unary prime spectra as developed in \cite[6.1, 6.2]{Du}, and hence will not have to deal with more sophisticated notions of spectrum, such as the one resulting from the total localization theory of \cite[6.3]{Du}. In particular, if $A$ is a generalized ring, $\Spec(A)$ is just the set of prime ideals in $|A|=A(1)$ (set of unary operations of $A$; cf. \cite[4.2.1]{Du}) endowed with the usual Zariski topology and the sheaf of generalized rings satisfying $\mathcal{O}(D(f))=A_f$ for all $f \in |A|$.
 
In addition, we would like to work over the base $\mathds{F}_1$. This means that every scheme is defined over $\Spec(\mathds{F}_1)$ and that every generalized ring $A$ is an $\mathds{F}_1$-algebra, i.e. it has a zero (\cite[6.5.6]{Du}). This implies in particular that the category $\M(A)$ of $A$-modules has a zero object and that we can talk about cokernels and kernels. For example, for a submodule $U$ of an $A$-module $M$, we define $M/U$ to be the cokernel of the inclusion map $U \hookrightarrow M$; this is nothing but the coequalizer of the inclusion $U \hookrightarrow M$ and the zero map $U\to M$. The same remarks apply to $\Q(X)$, the category of quasi-coherent modules on a generalized scheme $X$.

The notion of quasi-compact morphism/scheme \cite[6.5.15]{Du} gives rise to the usual notion of quasi-separated morphism/scheme: A morphism is quasi-separated if its diagonal is quasi-compact. As before, we will restrict to quasi-compact quasi-separated schemes, mainly because \cite[6.9.9]{EGAI} tells us that in this situation every quasi-coherent module is the directed colimit of its quasi-coherent submodules of finite type, which was an essential ingredient in our proof of Lemma \ref{iso}. The same proof works for generalized schemes, the only nontrivial ingredient being the fact that for a quasi-compact quasi-separated morphism $f$, the direct image functor $f_*$ preserves quasi-coherence. The proof in \cite[6.7.1]{EGAI} can easily be translated to generalized schemes: Instead of considering the kernel of a difference of two module homomorphisms, consider their equalizer. To mention a related example, in the computation of the structural sheaf of an affine scheme (\cite[1.4.1]{EGAI}), just replace the equations $f^{m_{ij}} (s_i - s_j) = 0$ with $f^{m_{ij}} s_i = f^{m_{ij}} s_j$. This equation then makes sense in a generalized ring, which does not come equipped with an addition operation, let alone a subtraction. This reveals a general recipe for the translation of proofs in \cite{EGAI} to generalized schemes: Write $x = y$ instead of $x - y = 0$. It is interesting to note in this context that oftentimes, the addition in a ring is not as ``fundamental'' as its multiplication and its unit are for the development of commutative algebra and algebraic geometry (cf. \cite[4.8]{Du}); this is precisely the reason why Durov's theory works so well. 

We hope that this brief discussion convinces the reader that at least those results from \cite{EGAI} which were used above in the proof of Theorem \ref{main}, as well as those which were needed in their proofs, are easily adaptable to generalized schemes.


\subsection{Translation of the proofs}

Now let us go through Sections $2$ and $3$ above in the generalized scheme setting.

We consider (cocomplete) symmetric monoidal categories and (cocontinuous) symmetric monoidal functors between them and for consistency, call them (cocomplete) tensor categories and (cocontinuous) tensor functors. For every generalized scheme $X$ the category of quasi-coherent modules $\Q(X)$ is a cocomplete tensor category with unit $|\O_X|$. Note that $\O_X$ itself is not a quasi-coherent module, but rather a sheaf of generalized rings. Every morphism of generalized schemes $f : Y \to X$ induces a cocontinuous tensor functor $f^* : \Q(X) \to \Q(Y)$, and our aim will be to show that as before, this implements an equivalence of categories 
\[\Hom(Y,X)\simeq \Hom_{c\otimes}\bigl(\Q(X),\Q(Y)\bigr)\]
when $X$ is quasi-compact and quasi-separated. 

We will not fix a base ring and do \emph{not} require our tensor categories to be linear. This will force us to find a more direct proof of the universal property of $\M(A)$ as a cocomplete tensor category, thereby generalizing Proposition \ref{modtens}. In order to do this, we will have to find a \emph{generalized} ring associated to a cocomplete tensor category $C$, playing the role of $\End(1_C)$ in this general setting. The following construction achieves this goal, and generalizes \cite[4.3.8]{Du}:

\begin{defi}
Let $C$ be a category with coproducts, and $M \in C$. Define the algebraic monad $\END(M)$ as follows: For every $n \in \mathds{N}$ let
 \[\END(M)(n) := \Hom(M,M^{\oplus n}).\]
For $t \in \END(M)(k)$ and $t_1,\dotsc,t_k \in \END(M)(n)$ define $t(t_1,\dotsc,t_k) \in \END(M)(n)$ as follows: By the universal property of the direct sum, the $k$ morphisms $t_i$ correspond to a morphism $(t_1,\dotsc,t_k) : M^{\oplus k} \to M^{\oplus n}$. Compose this with $t : M \to M^{\oplus k}$ to obtain a morphism $t(t_1,\dotsc,t_k) : M \to M^{\oplus n}$.

It is straightforward to check the cosimplicial identities (\cite[4.3.2]{Du}). Therefore, $\END(M)$ is, indeed, an algebraic monad.  
\end{defi}

\begin{lemma}
Let $C$ be a cocomplete tensor category. Then the algebraic monad $\END(1_C)$ is commutative, and hence a generalized ring.
\end{lemma}

\begin{proof}
Let us write $1:=1_C$. Letting $t \in \END(1)(n)$ and $t' \in \END(1)(m)$, we have to show (\cite[5.1.1]{Du}) that $a=b$ where
\[a=t(t'(x_{11},\dotsc,x_{1m}),\dotsc,t'(x_{n1},\dotsc,x_{nm})),\]
\[b=t'(t'(x_{11},\dotsc,x_{n1}),\dotsc,t'(x_{1m},\dotsc,x_{nm})).\]
Here, $(x_{ij})$ is some $n \times m$-matrix of elements in some $\END(1)$-module $X$. It suffices to take the universal example, which is $X=\END(1)(n \times m)$ together with the matrix $x_{ij} = \{(i,j)\}_{n \times m}$, i.e. $x_{ij} : 1 \to 1^{\oplus (n \times m)}$ is the inclusion of the summand with index $(i,j)$. Unwinding the definitions, $a$ is just the composition
\[\xymatrix@C=35pt{1 \ar[r]^-{t} & 1^{\oplus n} \ar[r]^-{t'^{\oplus n}} & (1^{\oplus m})^{\oplus n} \ar[r]^{\cong} & 1^{\oplus (n \times m)}.}\]
Similarly, $b$ factors as
\[\xymatrix@C=35pt{1 \ar[r]^-{t'} & 1^{\oplus m} \ar[r]^-{t^{\oplus m}} & (1^{\oplus n})^{\oplus m} \ar[r]^{\cong} & 1^{\oplus (n \times m)}.}\]
The following commutative diagram finishes the proof: 
\[\xymatrix@=18pt@R=40pt{(1^{\oplus n})^{\oplus m} \ar[rr]^{\cong} && 1^{\oplus n} \otimes 1^{\oplus m} && (1^{\oplus m})^{\oplus n} \ar[ll]_{\cong} \\
1^{\oplus m} \ar[u]^{t^{\oplus m}} \ar[r]^-{\cong} & 1 \otimes 1^{\oplus m} \ar[ur]^{t \otimes 1^{\oplus m}} & & 1^{\oplus n} \otimes 1 \ar[ul]_{1^{\oplus n} \otimes t'} & 1^{\oplus n} \ar[l]^-{\cong} \ar[u]_{t'^{\oplus n}} \\
1 \ar[rr]_{\cong} \ar[u]^{t'} && 1 \otimes 1 \ar[uu]|{t \otimes t'} \ar[ul]^{1 \otimes t'} \ar[ur]_{t \otimes 1} && 1  \ar[ll]^{\cong} \ar[u]_{t}}\]

Notice that for $n=m=1$ this is the usual proof that $\End(1)$ is a commutative monoid.
\end{proof}

We now state and prove the analogue of Proposition \ref{modtens}:

\begin{prop} \label{modtens2}
For a cocomplete tensor category $C$ and a generalized ring $A$, there is an equivalence of categories
\[\Hom_{c\otimes}(\M(A),C) \simeq \Hom(A,\END(1_C)).\]
\end{prop}
 
\begin{proof}
In the category of $A$-modules $\M(A)$ the unit is $|A|$, the free module of rank $1$, and the generalized ring $\END(|A|)$ turns out to be isomorphic to $A$:
\[\END(A)(n) = \Hom_{\M(A)}(|A|,|A|^{\oplus n}) = \Hom_{\M(A)}(A(1),A(n)) \cong A(n)\]
Since $\END$ is functorial, every cocontinuous tensor functor $\M(A) \to C$ thus induces a homomorphism of generalized rings $A \cong \END(|A|) \to \END(1_C)$.

Conversely, let $\alpha : A \to \END(1_C)$ be a homomorphism of generalized rings. Then for every $n$ we have a map $\alpha(n) : A(n) \to \End(1_C,1_C^{\oplus n})$. For $X \in C$ this endows the set $\Hom(1_C,X)$ with an $A$-module structure: For $t \in A(n)$ and $x_1,\dotsc,x_n : 1_C \to X$ define $t(x_1,\dotsc,x_n) : 1_C \to X$ to be the composition of $\alpha(n)(t)$ with the morphism $(x_1,\dotsc,x_n) : 1_C^{\oplus n} \to X$. We thus obtain a functor $\Hom(1_C,-) : C \to \M(A)$, which is clearly continuous. We claim that it has a left adjoint $1_C \otimes_A ?$.

On free modules, let $1_C \otimes_A A(X) := 1_C^{\oplus X}$. For a homomorphism of $A$-modules $\sigma : A(X) \to A(Y)$ we obtain a morphism $1_C^{\oplus X} \to 1_C^{\oplus Y}$ as follows: Since $A(X)$ is the free $A$-module on $X$, $\sigma$ corresponds to elements $t_x \in A(Y)$ for $x \in X$. Now apply $\alpha(Y)$ to get $1_C \to 1_C^{\oplus Y}$ for every $x$, i.e. a morphism $1_C^{\oplus X} \to 1_C^{\oplus Y}$. Then one easily checks that $1_C \otimes_A ?$ is a well-defined functor on the category of free $A$-modules. Now it is clear how to extend $1_C \otimes_A ?$ to arbitrary $A$-modules, since every $A$-module is the coequalizer of two maps between free $A$-modules (\cite[3.3.20]{Du}). Because of the defining adjunction, we do not have to check functoriality of $1_C \otimes_A ?$ on all $A$-modules.

Since $1_C \otimes_A ?$ is left adjoint, it is cocontinuous. Also, $1_C \otimes_A |A| = 1_C$ by construction, so that it preserves the units. In order to show that it preserves the tensor products, observe that $\Hom(1_C,-)$ is a lax tensor functor, which means that there are natural homomorphisms of $A$-modules
\[\Hom(1_C,X) \otimes \Hom(1_C,Y) \to \Hom(1_C,X \otimes Y),\]
which are induced by the tensor product in $C$ and $1_C \otimes 1_C \cong 1_C$. Then the adjunction endows $1_C \otimes_A ?$ with the structure of an oplax tensor functor, which means that there are natural morphisms
\[1_C \otimes_A (M \otimes_A N) \to (1_C \otimes_A M) \otimes (1_C \otimes_A N).\]
In order to show that this is an isomorphism for all $A$-modules $M,N$, a colimit argument reduces to the case $M=N=|A|$; but then it is trivial.
 
Thus, $1_C \otimes_A ?$ is a cocontinuous tensor functor. Now it is a routine exercise to show that the functors between $\Hom_{c\otimes}(\M(A),C)$ and $\Hom(A,\END(1_C))$, which we have just defined on objects, are in fact quasi inverse functors.
\end{proof}

We immediately deduce Corollary \ref{afftens} for generalized schemes, which means that we are done with affine schemes.

Of course, no modification of the category theoretic part of section 2.3 is necessary. We have already discussed that \cite[6.7.1]{EGAI} generalizes to the present setting, so Example \ref{directimage} still works. What about the notion of an immersion $i : Y \hookrightarrow Z$ of generalized schemes? Durov suggests a rather categorical one in \cite[6.5.21]{Du}, whose properties remain unclear. However, we only need open immersions, which behave as usual (\cite[6.5.7]{Du}), and closed immersions (\cite[6.5.23]{Du}), which are defined as follows:
 
A morphism of schemes $i : Y \to X$ is a \textit{closed immersion} if $i^\# : \O_X \to i_* \O_Y$ is a surjective homomorphism of sheaves of generalized rings. It follows that closed subschemes of $X$ correspond to strict quotients of $\O_X$. If $J$ denotes the kernel of $i^\#$, then $J$ is a generalized quasi-coherent ideal of $\O_X$, and there is a surjective homomorphism $\O_X/J \to i_* \O_Y$; it does not, however, have to be an isomorphism. This is also the reason why closed immersions don't have a closed image in general (\cite[6.5.19]{Du}). However, we only need to consider those closed immersions for which $\O_X/J \cong i_* \O_Y$ and $\O_X/J$ is a \textit{unary} $\O_X$-algebra (\cite[5.1.15, 6.5.14]{Du}).
 
These may be described as follows: Let $I \subseteq |\O_X|$ be a quasi-coherent ideal. The corresponding closed immersion $i : \Spec(\O_X/I) \hookrightarrow X$ is affine and in the case $X=\Spec(A)$ is induced by the canonical projection $A \to A/I$, where $I \subseteq |A|$ is an ideal and $A/I$ is defined by the usual universal property. Under the correspondence between unary $A$-algebras and algebra objects in $\M(A)$ (\cite[5.3.8]{Du}), the quotient $A/I$ corresponds to $|A|/I$. 
 
In the following, we only refer to closed immersions of this type. Now Corollaries \ref{immersion} and \ref{reduction} carry over to generalized schemes.

In the proof of Proposition \ref{fullfaith} we need for a generalized scheme $Y$ and a point $y \in Y$ the canonical morphism $\Spec(\O_{Y,y}) \to Y$. This is constructed as in \cite[2.5.1]{EGAI}: Reduce to $X = \Spec(A)$, in which case $y$ corresponds to a prime ideal $\mathfrak{p} \subseteq |A|$. We choose the morphism corresponding to the homomorphism of generalized rings $A \to A_{\mathfrak{p}}$ which is part of the defining universal property of this localization. In the next paragraph of the proof we need the following formula:
 
\begin{lemma}
For a morphism of generalized schemes $f : Y \to X$ and a quasi-coherent ideal $I \subseteq |\O_X|$ we have
\[\supp f^*(|\O_X|/I) = f^{-1}(\supp(|\O_X|/I)).\]
\end{lemma}

\begin{proof}
We can use the same proof as in \cite[Chapitre 0, 5.2.4.1]{EGAI}: One reduces to the claim that for a local homomorphism of generalized local rings (\cite[6.2.4]{Du}) $A \to B$ and an ideal $I \subseteq |A|$, one has $|A|/I \otimes_A B = 0$ if and only if one has $|A|/I=0$. Therefore, let us assume $|A|/I \neq 0$. Then $I \subseteq \mathfrak{m}_A$, and we have two epimorphisms $|A|/I \otimes_A B \twoheadrightarrow |A|/\mathfrak{m}_A \otimes_A B \twoheadrightarrow |B|/\mathfrak{m}_B \neq 0$. This proves that $|A|/I \otimes_A B$ is non-zero.
\end{proof}
 
The remainder of sections 3.1 and 3.2 also carries over directly to generalized schemes. In section 3.3 every occurrence of ``epimorphism'' has to be replaced by ``regular epimorphism'', which just means a surjective homomorphism. Recall that not every epimorphism of $A$-modules, where $A$ is a generalized ring, is surjective: The modules over $\mathds{Z}_{\geq 0}$ are precisely the commutative monoids (\cite[3.4.12 (a)]{Du}) and localizations such as $\mathds{N} \to \mathds{Z}$ provide non-surjective epimorphisms of commutative monoids. Apart from that, we don't have to modify section 3.3. Hence, every generalized local ring is good.
 
In Proposition \ref{map} we need the universal property of the canonical morphism $\Spec(\O_{Y,y}) \to Y$ in the context of generalized schemes, whose proof can be copied from \cite[2.5.3]{EGAI}. The argument at the end of the proof of Proposition \ref{map} has already been repeated above, in our discussion on Proposition \ref{fullfaith}. Finally, Theorems \ref{good} and \ref{main} follow as before.

Hence, we have proven:

\begin{thm} \label{main2}
Let $X,Y$ be generalized schemes over $\mathds{F}_1$, with $X$ quasi-compact and quasi-separated. Then, the functor $f \mapsto f^*$ implements an equivalence
\[\Hom(Y,X)\simeq \Hom_{c\otimes}\bigl(\Q(X),\Q(Y)\bigr).\]
\end{thm}

This actually contains our previous Theorem \ref{main} as a special case. It can also be applied to monoid schemes in the sense of \cite[Section 3]{CoCo}; see \cite[2.5]{PeLo} for the interpretation of monoid schemes as generalized schemes with zero.



\begin{thebibliography}{CHW}

\bibitem[Br1]{Br} M. Brandenburg, \emph{Tensorial schemes}, Preprint, \href{http://arxiv.org/abs/1110.6523v1}{arXiv:1110.6523v1}, 2011 
\bibitem[Br2]{Br2} M. Brandenburg, \emph{Rosenberg's Reconstruction Theorem (after Gabber)}, Preprint, \href{http://arxiv.org/abs/1310.5978v3}{arXiv:1310.5978v3}, 2013
\bibitem[CC]{CoCo} A. Connes, C. Consani, \textit{Schemes over $\mathds{F}_1$ and Zeta functions}, Preprint. \href{http://arxiv.org/abs/0903.2024v3}{arXiv:0903.2024v3}, 2009
\bibitem[CHW]{T} G. Corti\~{n}as, C. Haesemeyer, M. Walker, C. A. Weibel, \textit{Toric varieties, monoid schemes and $cdh$ descent}, Preprint, \href{http://arxiv.org/abs/1106.1389v1}{arXiv:1106.1389v1}, 2011
\bibitem[CJ]{ChJo} A. Chirvasitu, T. Johnson-Freyd, \textit{The fundamental pro-groupoid of an affine $2$-scheme}, Preprint, \href{http://arxiv.org/abs/1105.3104v4}{arXiv:1105.3104v4}, 2011
\bibitem[Dur]{Du} N. Durov, \textit{New Approach to Arakelov geometry}, Preprint,\\ \href{http://arxiv.org/abs/0704.2030}{arXiv:0704.2030v1}, 2007
\bibitem[Fre]{Fre} J. Fresan, \textit{Compactification projective de $\Spec(\mathds{Z})$ (d'apres Durov)}, Preprint, \href{http://arxiv.org/abs/0908.4059v1}{arXiv:0908.4059v1}, 2009
\bibitem[Gab]{Gab} P. Gabriel, \emph{Des cat\'{e}gories abeli\'{e}nnes}, Bull. Soc. Math. France 90, 1962, pp. 323--448.
\bibitem[GD]{EGAI} A. Grothendieck, J. Dieudonn\'{e}, \textit{\'{E}l\'{e}ments de g\'{e}om\'{e}trie alg\'{e}brique (r\'{e}dig\'{e}s avec la collaboration de Jean Dieudonn\'{e}): I. Le langage des sch\'{e}mas}, Grundlehren der Mathematischen Wissenschaften 166, Springer-Verlag, 2nd edition, 1971
\bibitem[Gil]{Gil} W. D. Gillam, \emph{Localization of ringed spaces}, Preprint, \href{http://arxiv.org/abs/1103.2139}{arXiv:1103.2139}, 2011
\bibitem[IK]{ImKe} G. B. Im, G. M. Kelly, \textit{A universal property of the convolution monoidal structure}, J. Pure Appl. Algebra, 43, 1986, No. 1, pp. 75--88
\bibitem[Ke1]{Kel2} G. M. Kelly, \textit{Doctrinal Adjunction}, Lecture Notes in Mathematics, Volume 420, 1974, pp. 257--280
\bibitem[Ke2]{Kel1} G. M. Kelly, \textit{Basic Concepts of Enriched Category Theory}, Reprints in Theory and Applications of Categories, No. 10, 2005, pp. 1--136
\bibitem[Lur]{Lur} J. Lurie, \textit{Tannaka Duality for Geometric Stacks}, Preprint,\\\href{http://arxiv.org/abs/math/0412266}{arXiv:math/0412266}, 2005
\bibitem[Mac]{Mac} S. Mac Lane, \textit{Categories for the Working Mathematician}, Second Edition, Graduate Texts in Mathematics, 5, Springer, 1998
\bibitem[PL]{PeLo} J. L. Pe\~{n}a, O. Lorscheid, \textit{Mapping $\mathds{F}_1$-land: An overview of geometries over the field with one element}, Preprint, \href{http://arxiv.org/abs/0909.0069}{arXiv:0909.0069v1}, 2009
\bibitem[Ros]{Ros} A. Rosenberg, \textit{Spectra of 'spaces' represented by abelian categories}, MPI Preprints Series 115, 2004.
\bibitem[SaR]{SaR} N. Saavedra-Rivano, \textit{Cat\'{e}gories Tannakiennes}, Lectures Notes in Mathematics, Vol. 265, Springer-Verlag, 1972
\bibitem[Sch1]{Sch1} D. Sch\"appi, \emph{A characterization of categories of coherent sheaves of certain algebraic stacks}, Preprint, \href{http://arxiv.org/abs/1206.2764}{arXiv:1206.2764}, 2012
\bibitem[Sch2]{Sch2} D. Sch\"appi, \emph{Ind-abelian categories and quasi-coherent sheaves}, Preprint, \href{http://arxiv.org/abs/1211.3678}{arXiv:1211.3678}, 2012
\bibitem[Str]{St} R. Street, \textit{Categorical structures}, Handbook of algebra, Vol. 1, pp. 529--577, North-Holland, 1996
\bibitem[Vez]{Vez} A. Vezzani, \textit{On the geometry over the field with one element}, Master Thesis, 2010, available at \href{http://users.mat.unimi.it/users/vezzani}{http://users.mat.unimi.it/users/vezzani}
\bibitem[Vis1]{V2} A. Vistoli, \textit{Intersection theory on algebraic stacks and on their moduli spaces}, Invent. math. 97, 1989, pp. 613--670
\bibitem[Vis2]{Vis} A. Vistoli, \textit{Notes on Grothendieck topologies, fibered categories and descent theory}, Preprint, \href{http://arxiv.org/abs/math/0412512}{arXiv:math/0412512v4}, 2007

\end{thebibliography}
\end{document}